\documentclass[11pt]{article}
\usepackage{amsfonts,amscd,amssymb, amsmath}

\allowdisplaybreaks[4]
\newtheorem{definition}{Definition}[section]

\newtheorem{proposition}[definition]{Proposition}
\newtheorem{corollary}[definition]{Corollary}
\newtheorem{remark}[definition]{Remark}

\newtheorem{theorem}[definition]{Theorem}
\newtheorem{example}[definition]{Example}

\def\rawo\lonra{\longrightarrow}

\def\ot{\otimes}

\allowdisplaybreaks[4]

\newenvironment{proof}{{\it Proof.}}{\hfill $ \square $ \vskip 4mm}

\begin{document}
\title{Twisted algebras and 
Rota-Baxter type operators
\thanks{Research partially supported by FWO-Vlaanderen (Flemish Fund for Scientific Research) 
within the research project ''Equivariant Brauer groups and Galois deformations''.  
The first author (F. P.) was also partially supported by 
a grant of the Romanian National 
Authority for Scientific Research, CNCS-UEFISCDI, 
project number PN-II-ID-PCE-2011-3-0635,  
contract nr. 253/5.10.2011.}}
\author
{Florin Panaite\\
Institute of Mathematics of the 
Romanian Academy\\ 
PO-Box 1-764, RO-014700 Bucharest, Romania\\
e-mail: Florin.Panaite@imar.ro\\
\and 
Freddy Van Oystaeyen\\
Department of Mathematics and Computer Science\\
University of Antwerp, Middelheimlaan 1\\
B-2020 Antwerp, Belgium\\
e-mail: fred.vanoystaeyen@uantwerpen.be}
\date{}
\maketitle
\begin{abstract}
We define the concept of weak pseudotwistor for an algebra $(A, \mu )$ in a monoidal category 
$\mathcal{C}$, as a morphism $T:A\otimes A\rightarrow A\otimes A$ in $\mathcal{C}$, satisfying some axioms 
ensuring that $(A, \mu \circ T)$ is also an algebra in $\mathcal{C}$. This concept generalizes 
the previous proposal called pseudotwistor and covers a number of exemples of twisted algebras 
that cannot  
be covered by pseudotwistors, mainly examples provided by Rota-Baxter operators and some of their 
relatives (such as Leroux's TD-operators and Reynolds operators).  By using weak pseudotwistors, 
we introduce 
an equivalence relation (called ''twist equivalence'') for algebras in a given monoidal category. 
\end{abstract}
\section{Introduction}
\setcounter{equation}{0}
${\;\;\;}$The concept of pseudotwistor (with a particular case called twistor) was introduced in \cite{lpvo} 
as a general device for twisting (or deforming) the multiplication of an algebra in a monoidal 
category, obtaining thus a new algebra structure on the same object (informally, we call ''twisted algebra'' an 
algebra that can be obtained by deforming the multiplication of a given algebra, maybe with the help of 
some data on the initial algebra).
Namely, if $A$ is an algebra 
with multiplication $\mu :A\ot A\rightarrow A$ in a monoidal category $\mathcal{C}$, 
a pseudotwistor for $A$ is a morphism $T:A\ot A\rightarrow A\ot A$ in $\mathcal{C}$, such that there 
exist two morphisms $\tilde{T}_1, \tilde{T}_2:A\ot A\ot A\rightarrow A\ot A\ot A$ in $\mathcal{C}$, 
called the companions of $T$, satisfying some axioms ensuring that $(A, \mu \circ T)$ is also an algebra 
in $\mathcal{C}$. There are many classes of examples of such pseudotwistors. 
The one that was the starting point of \cite{lpvo} is provided by a twisted tensor product of algebras $A\ot _RB$ 
(as in \cite{cap}, \cite{vandaele}), which is a twisting by a twistor of the ordinary tensor product of 
algebras $A\ot B$. Another class of examples is provided by braidings: if $c$ is a braiding on a 
monoidal category $\mathcal{C}$, then $c^2_{A, A}$ is a pseudotwistor for every algebra $A$ in $\mathcal{C}$. 
The so-called Fedosov product provides another class of examples. Finally, if $H$ is a bialgebra 
and $\sigma :H\ot H\rightarrow k$ is a normalized and convolution invertible left 2-cocycle, 
one can consider the twisted algebra $_{\sigma }H$, which is the associative algebra structure 
on $H$ with multiplication $a*b=\sigma (a_1, b_1)a_2b_2$, for all $a, b\in H$; 
it turns out that this multiplication 
is afforded by the pseudotwistor 
\begin{eqnarray}
&&T:H\ot H\rightarrow H\ot H, \;\;\;T(a\ot b)=\sigma (a_1, b_1)a_2\ot b_2, \;\;\;\forall \;\;a, b\in H. 
\label{cucu}
\end{eqnarray}

An indication that a more general concept than pseudotwistors might exist is already implicit in 
two of the examples given above. For a twisted algebra of the type $_{\sigma }H$, the multiplication 
$*$ is associative even if $\sigma $ is not convolution invertible, but the map $T$ given by 
(\ref{cucu}) is no longer a pseudotwistor in this case. Also, if $c$ is only a pre-braiding on a monoidal 
category $\mathcal{C}$ (i.e. we do not assume the invertibility of the morphisms $c_{-, -}$)  
and $(A, \mu )$ is an algebra in $\mathcal{C}$, then $(A, \mu \circ c^2_{A, A})$ is still an algebra 
in the category but $c^2_{A, A}$ is no longer a pseudotwistor. 

However, we were led to a generalization of pseudotwistors by looking at another class of 
examples of twisted algebras, provided by Rota-Baxter operators. If $(A, \mu )$ is an associative algebra over 
a field $k$, with notation $\mu (a\ot b)=ab$, for $a, b\in A$, and $\lambda \in k$ is a fixed 
element, a linear map $R:A\rightarrow A$ is called a Rota-Baxter operator of 
weight $\lambda $ if it satisfies the relation
$R(a)R(b)=R(R(a)b+aR(b)+\lambda ab)$,  for all $a, b\in A$. 
It is well-known that the new multiplication $*_{\lambda }$ on $A$ defined by 
$a*_{\lambda }b=R(a)b+aR(b)+\lambda ab$ is associative. Also, it is by now well-known that 
Rota-Baxter operators represent a part of the algebraic component of the Connes-Kreimer approach 
to renormalization (see \cite{conneskreimer}, \cite{ladder}, \cite{kreimer} and references therein). 
If we define the linear map $T:A\ot A\rightarrow A\ot A$, 
$T(a\ot b)=R(a)\ot b+a\ot R(b)+\lambda a\ot b$, for all $a, b\in A$, 
then the associative multiplication 
$*_{\lambda }$ may be written as $*_{\lambda }=\mu \circ T$, but T is far from being a pseudotwistor. 

Motivated by all these examples, we introduce the following concept. Assume that $(A, \mu )$ is an 
algebra in a monoidal category $\mathcal{C}$, $T:A\otimes
A\rightarrow A\otimes A$ and 
$\mathcal{T}:A\otimes A\otimes A \rightarrow A\otimes A\otimes A$ are morphisms in 
$\mathcal{C}$ such that:
\begin{eqnarray*}
&&T\circ (id_A\otimes (\mu \circ T))=(id_A\otimes \mu )\circ \mathcal{T},\\
&&T\circ ((\mu \circ T)\otimes id_A)=(\mu \otimes id_A)\circ \mathcal{T}.
\end{eqnarray*}
Then $(A, \mu \circ T)$ is also an algebra in ${\mathcal{C}}$,
denoted by $A^T$; the morphism $T$ is called a weak 
pseudotwistor for $A$ and the morphism $\mathcal{T}$
is called the weak companion of $T$. It turns out that all the above-mentioned examples of 
deformed associative multiplications are afforded by weak pseudotwistors, and we provide as well 
some other examples, coming especially from Rota-Baxter type operators (Reynolds operators, 
Leroux's TD-operators etc). We present also some general properties of weak pseudotwistors. 

A new class of weak pseudotwistors, coming from so-called Rota-Baxter systems, may be found in the 
recent paper \cite{brzez}. In fact, as noted by the referee, some of the operators presented in Section 
\ref{section3} of our paper (including the TD-operators) are examples of Rota-Baxter systems, which 
gives an alternative way of proving that they yield weak pseudotwistors. 

In the last section we use weak pseudotwistors in order to introduce an equivalence relation for 
algebras in a monoidal category $\mathcal{C}$: if $A$ and $B$ are two such algebras, 
we say that $A$ and $B$ are 
twist equivalent (and write $A\equiv _tB$) if there exists an invertible weak pseudotwistor 
$T$ for $A$, with invertible weak companion $\mathcal{T}$, such that $A^T$ and $B$ are 
isomorphic as algebras. For example, if $A\ot _RB$ is a twisted tensor product of algebras with 
bijective twisting map $R$, then $A\ot _RB\equiv _tA\ot B$. 

Unless otherwise specified, the (co)algebras that will appear in this paper are {\em not} 
supposed to be (co)unital; if $A$ is an associative algebra over a field $k$ we usually denote 
the multiplication of $A$ by $\mu :A\ot A\rightarrow A$, $\mu (a\ot b)=ab$, for all $a, b\in A$. 
For the composition of two morphisms $f$ and $g$ we write either $g\circ f$ or simply $gf$. 
For unexplained terminology we refer to \cite{k}. 

\section{Weak pseudotwistors}
\setcounter{equation}{0}
${\;\;\;}$We recall the concept of pseudotwistor introduced in \cite{lpvo} 
(the version for nonunital algebras). 
\begin{definition} 
Let $(\mathcal{C}, \ot )$ be a strict monoidal category, $A$ an algebra in
$\mathcal{C}$ with multiplication $\mu :A\ot A\rightarrow A$ and  $T:A\otimes
A\rightarrow A\otimes A$ a morphism in $\mathcal{C}$. Assume that there exist two morphisms 
$\tilde{T}_1, \tilde{T}_2:A\otimes A\otimes A \rightarrow A\otimes A\otimes A$ in
$\mathcal{C}$ such that:
\begin{eqnarray*}
&&T\circ (id_A\otimes \mu )=(id_A\otimes \mu )\circ \tilde{T}_1\circ (T\otimes id_A),\\
&&T\circ (\mu \otimes id_A)=(\mu \otimes id_A)\circ \tilde{T}_2\circ (id_A\otimes T),\\
&&\tilde{T}_1\circ (T\otimes id_A)\circ (id_A\otimes T)=
\tilde{T}_2\circ (id_A\otimes T)\circ (T\otimes id_A). 
\end{eqnarray*}
Then $(A, \mu \circ T)$ is also an algebra in ${\mathcal{C}}$,
denoted by $A^T$. The morphism $T$ is called a
{\em pseudotwistor} and the two morphisms $\tilde{T}_1$,
$\tilde{T}_2$ are called the {\em companions} of $T$.
\end{definition}

We recall from \cite{psvo} the categorical analogue of the concept of $R$-matrix 
introduced in \cite{borcherds1}, \cite{borcherds2} (also the version for nonunital algebras).
\begin{proposition} 
Let $(\mathcal{C}, \ot )$ be a strict monoidal category, $A$ an algebra in 
$\cal C$ with multiplication $\mu :A\ot A\rightarrow A$ and $T:A\otimes A\rightarrow  
A\otimes A$ a morphism in $\cal C$. 
Assume that there exist two morphisms 
$\overline{T}_1, \overline{T}_2:A\otimes A\otimes A  
\rightarrow A\otimes A\otimes A$ in $\cal C$ such that:  
\begin{eqnarray*}
&&T\circ (id_A\otimes \mu )=(id_A\otimes \mu )\circ (T\otimes id_A)\circ \overline{T}_1, \\
&&T\circ (\mu \otimes id_A)=(\mu \otimes id_A)\circ (id_A\otimes T)\circ \overline{T}_2, \\
&&(T\otimes id_A)\circ \overline{T}_1\circ (id_A\otimes T)= 
(id_A\otimes T)\circ \overline{T}_2\circ (T\otimes id_A). 
\end{eqnarray*}
Then $(A, \mu \circ T)$ is also an algebra in $\mathcal{C}$,   
denoted by $A^T$. The morphism $T$ is called an {\sl $R$-matrix} and  
the two morphisms $\overline{T}_1$, $\overline{T}_2$ are called the 
{\sl companions} of $T$. 
\end{proposition}

We introduce now a common generalization of these two concepts. 
\begin{theorem} \label{main}
Let $(\mathcal{C}, \ot )$ be a strict monoidal category, $A$ an algebra in
$\mathcal{C}$ with multiplication $\mu :A\ot A\rightarrow A$ and  $T:A\otimes
A\rightarrow A\otimes A$ a morphism in $\mathcal{C}$. Assume that there exists a morphism
$\mathcal{T}:A\otimes A\otimes A \rightarrow A\otimes A\otimes A$ in
$\mathcal{C}$ such that:
\begin{eqnarray}
&&T\circ (id_A\otimes (\mu \circ T))=(id_A\otimes \mu )\circ \mathcal{T},
\label{wp1} \\
&&T\circ ((\mu \circ T)\otimes id_A)=(\mu \otimes id_A)\circ \mathcal{T}.
\label{wp2} 
\end{eqnarray}
Then $(A, \mu \circ T)$ is also an algebra in ${\mathcal{C}}$,
denoted by $A^T$. The morphism $T$ is called a {\em weak 
pseudotwistor} and the morphism $\mathcal{T}$
is called the {\em weak companion} of $T$.
\end{theorem}
\begin{proof} We compute:
\begin{eqnarray*}
(\mu \circ T)\circ ((\mu \circ T)\ot id_A)&=&\mu \circ T\circ ((\mu \circ T)\ot id_A)\\
&\overset{(\ref{wp2})}{=}&\mu \circ (\mu \ot id_A)\circ \mathcal{T}\\
&=&\mu \circ (id_A\ot \mu )\circ \mathcal{T}\\
&\overset{(\ref{wp1})}{=}&(\mu \circ T)\circ (id_A\ot (\mu \circ T)),
\end{eqnarray*}
finishing the proof.
\end{proof}
\begin{remark} \label{echival}
If $T$ is a pseudotwistor with companions $\tilde{T}_1$,
$\tilde{T}_2$ on an algebra $A$, then $T$ is also a weak pseudotwistor, with weak companion 
$\mathcal{T}=\tilde{T}_1\circ (T\otimes id_A)\circ (id_A\otimes T)=
\tilde{T}_2\circ (id_A\otimes T)\circ (T\otimes id_A)$.

Conversely, an {\em invertible} weak pseudotwistor $T$ with weak companion $\mathcal{T}$ on an 
algebra $A$ is a pseudotwistor, with companions 
$\tilde{T}_1=\mathcal{T}\circ (id_A\otimes T^{-1})\circ (T^{-1}\otimes id_A)$ and 
$\tilde{T}_2=\mathcal{T}\circ (T^{-1}\otimes id_A)\circ (id_A\otimes T^{-1})$.

If $T$ is an $R$-matrix with companions $\overline{T}_1$, $\overline{T}_2$ on an algebra $A$, 
then $T$ is a weak pseudotwistor, with weak companion
$\mathcal{T}=(T\otimes id_A)\circ \overline{T}_1\circ (id_A\otimes T)=
(id_A\otimes T)\circ \overline{T}_2\circ (T\otimes id_A)$.

Conversely, an {\em invertible} weak pseudotwistor $T$ with weak companion $\mathcal{T}$ on an 
algebra $A$ is an $R$-matrix, with companions 
$\overline{T}_1=(T^{-1}\otimes id_A)\circ \mathcal{T}\circ (id_A\otimes T^{-1})$ and 
$\overline{T}_2=(id_A\otimes T^{-1})\circ \mathcal{T}\circ (T^{-1}\otimes id_A)$.
\end{remark}
\begin{example}
Let $A$ be an associative unital algebra with unit $1_A$ over a field $k$. In \cite{lebed} 
the following linear map was considered:
\begin{eqnarray*}
&&T:A\ot A\rightarrow A\ot A, \;\;\;T(a\ot b)=1_A\ot ab, \;\;\;\forall \;a, b\in A. 
\end{eqnarray*}
The associativity of the multiplication of $A$ is equivalent to the fact that this map $T$ is a 
Yang-Baxter operator, cf. \cite{lebed}. 

One can check, by a direct computation, that $T$ is a pseudotwistor (in particular, a weak pseudotwistor)  
with companions $\tilde{T}_1=id_A\ot T$ and $\tilde{T}_2(a\ot b\ot c)=1_A\ot 1_A\ot abc$, 
for all $a, b, c\in A$, and obviously $A^T=A$. 
\end{example}
\begin{remark}
It is possible to have an associative algebra $A$ over a field $k$, a linear map $T:A\ot A\rightarrow A\ot A$ 
that is {\em not} a weak pseudotwistor but such that $\mu \circ T$ is associative. Indeed, for a 
given associative algebra $A$, the linear map $T:A\ot A\rightarrow A\ot A$, $T(a\ot b)=b\ot a$, 
for all $a, b\in A$, has the property that $\mu \circ T$ is associative (of course, it is just the 
multiplication of $A^{op}$) but, in general, $T$ is not a weak pseudotwistor (although it may be, for some 
particular algebras $A$). We give a concrete example of an associative algebra $A$ for which $T$ 
is not a weak pseudotwistor. We take $A$ to be a $2$-dimensional associative algebra over a field $k$ 
with a linear basis $\{x, y\}$ with multiplication table $x\cdot x=x\cdot y=y\cdot x=y\cdot y=x$. 
We claim that there exists no linear map $\mathcal{T}:A\otimes A\otimes A \rightarrow A\otimes A\otimes A$ 
such that $T\circ (id_A\otimes (\mu \circ T))(y\ot x\ot x)=(id_A\otimes \mu )\circ \mathcal{T}(y\ot x\ot x)$. 
Suppose that such a map exists. We have $T\circ (id_A\otimes (\mu \circ T))(y\ot x\ot x)=x\ot y$. 
If we write $\mathcal{T}(y\ot x\ot x)=a_1x\ot x\ot x+a_2x\ot x\ot y+a_3x\ot y\ot x+a_4x\ot y\ot y
+a_5y\ot x\ot x+a_6y\ot x\ot y+a_7y\ot y\ot x+a_8y\ot y\ot y$, 
with $a_1, ... , a_8\in k$, then we have $(id_A\otimes \mu )\circ \mathcal{T}(y\ot x\ot x)=
(a_1+a_2+a_3+a_4)x\ot x+(a_5+a_6+a_7+a_8)y\ot x$, and this element cannot be equal to $x\ot y$. 
\end{remark}

Let $(\mathcal{C}, \ot )$ be a strict monoidal category and $T_{X, Y}:X\ot Y\rightarrow X\ot Y$ a family 
of natural morphisms in $\mathcal{C}$ such that, for all $X, Y, Z\in \mathcal{C}$, we have:
\begin{eqnarray}
&&T_{X\ot Y, Z}\circ (T_{X, Y}\ot id_Z)=T_{X, Y\ot Z}\circ (id_X\ot T_{Y, Z}). \label{laycle}
\end{eqnarray}
It was proved in \cite{psvo} that if for all $X, Y\in \mathcal{C}$ the morphism $T_{X, Y}$ is an 
isomorphism, then, for every algebra $(A, \mu )$ in $\mathcal{C}$, the morphism 
$T_{A, A}:A\ot A\rightarrow A\ot A$ is a pseudotwistor. If we do not assume the invertibility of the 
morphisms $T_{X, Y}$, then $T_{A, A}$ is no longer a pseudotwistor. 
\begin{proposition}\label{weaklaycle}
$T_{A, A}$ is a weak pseudotwistor. 
\end{proposition}
\begin{proof}
The identity (\ref{laycle}) for $X=Y=Z=A$ becomes 
$T_{A\ot A, A}\circ (T_{A, A}\ot id_A)=T_{A, A\ot A}\circ (id_A\ot T_{A, A})$;  
we denote by $\mathcal{T}$ this morphism from $A\ot A\ot A$ to $A\ot A\ot A$. The naturality of $T$ 
implies that 
$T_{A, A}\circ (id_A\ot \mu )=(id_A\ot \mu )\circ T_{A, A\ot A}$ and 
$T_{A, A}\circ (\mu \ot id_A)=(\mu \ot id_A)\circ T_{A\ot A, A}$, 
hence:
\begin{eqnarray*}
&&T_{A, A}\circ (id_A\ot (\mu \circ T_{A, A}))=(id_A\ot \mu )\circ T_{A, A\ot A}\circ (id_A\ot T_{A, A})
=(id_A\ot \mu )\circ \mathcal{T}, \\
&&T_{A, A}\circ ((\mu \circ T_{A, A})\ot id_A)=(\mu \ot id_A)\circ T_{A\ot A, A}\circ (T_{A, A}\ot id_A)=
(\mu \ot id_A)\circ \mathcal{T}, 
\end{eqnarray*}
proving that $T_{A, A}$ is a weak pseudotwistor with weak companion $\mathcal{T}$. 
\end{proof}
\begin{definition}(\cite{cap}, \cite{vandaele}) Let $(\mathcal{C}, \ot )$ be a strict monoidal category and 
$(A, \mu _A)$, $(B, \mu _B)$ 
two algebras in $\mathcal{C}$. A {\em twisting map} between $A$ and $B$ is a morphism $R:B\ot A 
\rightarrow A\ot B$ in $\mathcal{C}$ satisfying the conditions: 
\begin{eqnarray}
&&R\circ (id_B\ot \mu _A)=(\mu _A\ot id_B)\circ (id_A\ot R)\circ (R\ot id_A),
\label{twm1}\\
&&R\circ (\mu _B\ot id_A)=(id_A\ot \mu _B)\circ (R\ot id_B)\circ (id_B\ot R).
\label{twm2}
\end{eqnarray}
\end{definition}

The next result generalizes \cite{lpvo}, Theorem 6.6 (iii). 
\begin{proposition}\label{2twmaps}
Let $(\mathcal{C}, \ot )$ be a strict monoidal category and 
$(A, \mu )$ an algebra in $\mathcal{C}$. Assume that $Q, P:A\ot A\rightarrow A\ot A$ are 
two twisting maps between $A$ and itself, such that:
\begin{eqnarray}
&&(P\ot id_A)\circ (id_A\ot P)\circ (P\ot id_A)=(id_A\ot P)\circ (P\ot id_A)\circ (id_A\ot P), \label{(6.10)}\\
&&(Q\ot id_A)\circ (id_A\ot Q)\circ (Q\ot id_A)=(id_A\ot Q)\circ (Q\ot id_A)\circ (id_A\ot Q), \label{(6.11)}\\
&&(P\ot id_A)\circ (id_A\ot P)\circ (Q\ot id_A)=(id_A\ot Q)\circ (P\ot id_A)\circ (id_A\ot P), \label{(6.12)}\\
&&(Q\ot id_A)\circ (id_A\ot P)\circ (P\ot id_A)=(id_A\ot P)\circ (P\ot id_A)\circ (id_A\ot Q). \label{(6.13)}
\end{eqnarray}
Then $T:=Q\circ P:A\ot A\rightarrow A\ot A$ is a weak pseudotwistor. 
\end{proposition}
\begin{proof}
Because of (\ref{(6.10)}) and (\ref{(6.11)}), we have the following equality: 
\begin{eqnarray*}
&&(Q\ot id_A)\circ (id_A\ot Q)\circ (Q\ot id_A)\circ (id_A\ot P)\circ (P\ot id_A)\circ (id_A\ot P)\\
&&\;\;\;\;\;\;\;\;\;\;\;\;\;\;\;\;=(id_A\ot Q)\circ (Q\ot id_A)\circ (id_A\ot Q)\circ 
(P\ot id_A)\circ (id_A\ot P)\circ (P\ot id_A). 
\end{eqnarray*}
This morphism from $A\ot A\ot A$ to $A\ot A\ot A$ will be denoted by $\mathcal{T}$. Now we compute:
\begin{eqnarray*}
T\circ (id_A\ot (\mu \circ T))&=&Q\circ P\circ (id_A\ot \mu )\circ (id_A\ot Q)\circ (id_A\ot P)\\
&\overset{(\ref{twm1})}{=}&Q\circ (\mu \ot id_A)\circ (id_A\ot P)\circ (P\ot id_A)\circ (id_A\ot Q)
\circ (id_A\ot P)\\
&\overset{(\ref{(6.13)})}{=}&Q\circ (\mu \ot id_A)\circ (Q\ot id_A)\circ (id_A\ot P)\circ (P\ot id_A)
\circ (id_A\ot P)\\
&\overset{(\ref{twm2})}{=}&(id_A\ot \mu )\circ (Q\ot id_A)\circ (id_A\ot Q)\circ (Q\ot id_A)\circ 
(id_A\ot P)\\
&&\circ (P\ot id_A)\circ (id_A\ot P)\\
&=&(id_A\ot \mu )\circ \mathcal{T}, 
\end{eqnarray*}
\begin{eqnarray*}
T\circ ((\mu \circ T)\ot id_A)&=&Q\circ P\circ (\mu \ot id_A)\circ (Q\ot id_A)\circ (P\ot id_A)\\
&\overset{(\ref{twm2})}{=}&Q\circ (id_A\ot \mu )\circ (P\ot id_A)\circ (id_A\ot P)\circ (Q\ot id_A)
\circ (P\ot id_A)\\
&\overset{(\ref{(6.12)})}{=}&Q\circ (id_A\ot \mu )\circ (id_A\ot Q)\circ (P\ot id_A)\circ (id_A\ot P)
\circ (P\ot id_A)\\
&\overset{(\ref{twm1})}{=}&(\mu \ot id_A)\circ (id_A\ot Q)\circ (Q\ot id_A)\circ (id_A\ot Q)\circ 
(P\ot id_A)\\
&&\circ (id_A\ot P)\circ (P\ot id_A)\\
&=&(\mu \ot id_A)\circ \mathcal{T}, 
\end{eqnarray*}
so $T$ is a weak pseudotwistor with weak companion $\mathcal{T}$. 
\end{proof}
\begin{corollary}
Let $(\mathcal{C}, \ot )$ be a strict monoidal category and $c_{X, Y}:X\ot Y\rightarrow Y\ot X$ a 
pre-braiding on $\mathcal{C}$ (that is $c$ satisfies all the axioms of a braiding as in \cite{k} except 
for the fact that we do not require $c_{X, Y}$ to be invertible). Let $(A, \mu )$ be an algebra in 
$\mathcal{C}$. Then $c^2_{A, A}=c_{A, A}\circ c_{A, A}:A\ot A\rightarrow A\ot A$ is a 
weak pseudotwistor.
\end{corollary}
\begin{proof}
It is either a consequence of Proposition \ref{weaklaycle}, by noting that the family of natural 
morphisms $T_{X, Y}:=c_{Y, X}\circ c_{X, Y}$ satisfies (\ref{laycle}), or a consequence of Proposition 
\ref{2twmaps}, applied to the twisting maps $Q=P=c_{A, A}$.  
\end{proof}

\begin{example} \label{cocycles}
Let $H$ be a bialgebra over a field $k$, with multiplication $\mu :H\ot H\rightarrow H$, 
$\mu (h\ot h')=hh'$, for $h, h'\in H$, and comultiplication $\Delta :H\rightarrow H\ot H$, for which 
we use a Sweedler-type notation $\Delta (h)=h_1\ot h_2$, for $h\in H$. Let 
$\sigma :H\ot H\rightarrow k$ be a left 2-cocycle, that is we have 
$\sigma (a_1, b_1)\sigma (a_2b_2, c)=\sigma (b_1, c_1)\sigma (a, b_2c_2)$, for all $a, b, c\in H$. 
It is well-known that, if we define a new multiplication on $H$, by $a*b=\sigma (a_1, b_1)a_2b_2$, 
for all $a, b\in H$, then this multiplication is associative and the new algebra structure on $H$ 
is denoted by $_{\sigma }H$. 

Define the linear map $T:H\ot H\rightarrow H\ot H$, $T(a\ot b)=\sigma (a_1, b_1)a_2\ot b_2$, for all 
$a, b\in H$. It was proved in \cite{lpvo} that, if $H$ is unital and counital and $\sigma $ is convolution 
invertible with inverse $\sigma ^{-1}$, then $T$ is a pseudotwistor, with companions 
$\tilde{T}_1, \tilde{T}_2:H\ot H\ot H\rightarrow H\ot H\ot H$, 
$\tilde{T}_1(a\otimes b\otimes c)=\sigma ^{-1}(a_1, b_1)\sigma (a_2,
b_2c_1)a_3\otimes b_3\otimes c_2$ and $\tilde{T}_2(a\otimes b\otimes c)=
\sigma ^{-1}(b_1, c_1)\sigma (a_1b_2, c_2)a_2\otimes b_3\otimes c_3$.

If $\sigma $ is {\em not} convolution invertible, $T$ is no longer a pseudotwistor. However, $T$ is a 
weak pseudotwistor, with weak companion $\mathcal{T}:H\ot H\ot H\rightarrow H\ot H\ot H$,
$\mathcal{T}(a\ot b\ot c)=\sigma (b_1, c_1)\sigma (a_1, b_2c_2)a_2\ot b_3\ot c_3=
\sigma (a_1, b_1)\sigma (a_2b_2, c_1)a_3\ot b_3\ot c_2$, 
for all $a, b, c\in H$,
as one can easily check. 

There exist ''mirror versions'' of these facts. Namely, let $\tau :H\ot H\rightarrow k$ be a right 2-cocycle, i.e. 
$\tau $ satisfies the condition 
$\tau (a_1b_1, c)\tau (a_2, b_2)=\tau (a, b_1c_1)\tau (b_2, c_2)$, for all $a, b, c\in H$. 
If we define a new multiplication on $H$ by $a*b=a_1b_1\tau (a_2, b_2)$, for all $a, b\in H$, 
then this multiplication is associative and the new algebra structure is denoted by $H_{\tau }$. 
This multiplication is afforded by the weak pseudotwistor 
$D:H\ot H\rightarrow H\ot H$, $D(a\ot b)=a_1\ot b_1\tau (a_2, b_2)$, whose weak companion is the 
linear map $\mathcal{D}:H\ot H\ot H\rightarrow H\ot H\ot H$, $\mathcal{D}(a\ot b\ot c)=
a_1\ot b_1\ot c_1\tau (a_2, b_2c_2)\tau (b_3, c_3)=
a_1\ot b_1\ot c_1\tau (a_2b_2, c_2)\tau (a_3, b_3)$.
\end{example}
\begin{example}
Let $A$ be an associative algebra with unit $1$ over a field $k$ and $\lambda , \theta , \nu \in k$ some 
fixed scalars. In \cite {dn} the following linear map was considered:
\begin{eqnarray*}
&&T:A\ot A\rightarrow A\ot A, \;\;\;T(a\ot b)=\lambda ab\ot 1+\theta 1\ot ab -\nu a\ot b, \;\;\;
\forall \;a, b\in A. 
\end{eqnarray*}
Then one can easily check that $T$ is a weak pseudotwistor with weak companion 
$\mathcal{T}:A\ot A\ot A\rightarrow A\ot A\ot A$, 
$\mathcal{T}(a\ot b\ot c)=(\lambda +\theta -\nu )(\lambda abc\ot 1\ot 1+\theta 1\ot 1\ot abc -
\nu a\ot b\ot c)$, for all $a, b, c\in A$. 
\end{example}
\begin{proposition}\label{bimod}
Let $(A, \mu )$ be an associative algebra. We regard $A\ot A$ as an $A$-bimodule in the usual way: 
$a\cdot (b\ot c)\cdot d=ab\ot cd$, for all $a, b, c, d\in A$. Let $\delta :A\rightarrow A\ot A$ be a linear map, 
with Sweedler-type notation $\delta (a)=a_1\ot a_2$, 
that is a morphism of $A$-bimodules, i.e. 
\begin{eqnarray}
&&\delta (ab)=ab_1\ot b_2, \label{bim1} \\
&&\delta (ab)=a_1\ot a_2b, \label{bim2}
\end{eqnarray}
for all $a, b\in A$. 
Define the linear map $T:A\ot A\rightarrow A\ot A$, $T(a\ot b)=b_1a\ot b_2$. Then $T$ satisfies the 
following relations (with standard notation for $T_{12}$, $T_{13}$, $T_{23}$):
\begin{eqnarray}
&&T\circ (id_A\ot \mu )=(id_A\ot \mu )\circ T_{12}, \label{relatie1}\\
&&T\circ (\mu \ot id_A)=(\mu \ot id_A)\circ T_{13}, \label{relatie2}\\
&&T_{12}\circ T_{23}=T_{13}\circ T_{12}.  \label{relatie3}
\end{eqnarray}
Consequently, $T$ is a weak pseudotwistor 
with weak companion $\mathcal{T}=T_{12}\circ T_{23}=T_{13}\circ T_{12}$, and the new 
multiplication on $A$ defined by $a*b=b_1ab_2$ is associative. 
\end{proposition}
\begin{proof} We compute:
\begin{eqnarray*}
T(a\ot bc)&=&(bc)_1a\ot (bc)_2\\
&\overset{(\ref{bim2})}{=}&b_1a\ot b_2c
=((id_A\ot \mu )\circ T_{12})(a\ot b\ot c), 
\end{eqnarray*}
\begin{eqnarray*}
&&T(ab\ot c)=c_1ab\ot c_2=((\mu \ot id_A)\circ T_{13})(a\ot b\ot c), 
\end{eqnarray*}
\begin{eqnarray*}
(T_{12}\circ T_{23})(a\ot b\ot c)&=&T_{12}(a\ot c_1b\ot c_2)\\
&=&(c_1b)_1a\ot (c_1b)_2\ot c_2\\
&\overset{(\ref{bim1})}{=}&c_1b_1a\ot b_2\ot c_2\\
&=&T_{13}(b_1a\ot b_2\ot c)\\
&=&(T_{13}\circ T_{12})(a\ot b\ot c), 
\end{eqnarray*}
finishing the proof.
\end{proof}
\begin{remark}
The definition of the multiplication $*$ in Proposition \ref{bimod} is inspired by 
the result of Aguiar from \cite{aguiar}, showing that if $(A, \mu , \Delta )$ is an 
infinitesimal bialgebra (i.e. $\Delta :A\rightarrow A\ot A$ is a coassociative derivation) 
then the new product on $A$ defined by $a\circ b=b_1ab_2$ is pre-Lie.  
\end{remark}
\begin{remark}
The referee suggested the following extension of Proposition \ref{bimod}. Consider the data $(A, \nu , \sigma, \delta )$, 
where $A$ is an associative algebra, $\nu , \sigma :A\rightarrow A$ are algebra automorphisms and 
$\delta :A\rightarrow A\otimes A$ is an $A$-bimodule map, where the $A$-bimodule structure of 
$A\otimes A$ is now twisted by $\nu $ and $\sigma $, that is 
\begin{eqnarray*}
&&a\cdot (b\otimes c)\cdot d=\nu (a)b\otimes c\sigma (d), \;\;\;\forall \;a, b, c, d\in A.
\end{eqnarray*}
By using the same Sweedler-type notation for $\delta $, namely $\delta (a)=a_1\otimes a_2$, define the linear map 
$T:A\otimes A\rightarrow A\otimes A$, $T(a\otimes b)=\nu ^{-1}(b_1)a\otimes \sigma ^{-1}(b_2)$. Then $T$ 
satisfies the relations (\ref{relatie1})-(\ref{relatie3}), hence it is also a weak pseudotwistor with weak companion 
$\mathcal{T}=T_{12}\circ T_{23}=T_{13}\circ T_{12}$, and the new 
multiplication defined on $A$ by $a*b=\nu ^{-1}(b_1)a\sigma ^{-1}(b_2)$ is associative. 
\end{remark}
\begin{proposition}\label{taretwopstw}
Let $(\mathcal{C}, \ot )$ be a strict monoidal category and $(A, \mu )$ an algebra in
$\mathcal{C}$. Assume that $T$ and $D$ are two weak pseudotwistors for $A$, with weak 
companions $\mathcal{T}$ and respectively $\mathcal{D}$, such that the following conditions are satisfied:
\begin{eqnarray}
&&D\circ (id_A\ot (\mu \circ T\circ D))=(id_A\ot (\mu \circ T))\circ \mathcal{D}, \label{tarecomp1}\\
&&D\circ ((\mu \circ T\circ D)\ot id_A)=((\mu \circ T)\ot id_A)\circ \mathcal{D}. \label{tarecomp2}
\end{eqnarray}
Then $T\circ D$ is a weak pseudotwistor for $A$, with weak companion $\mathcal{T}\circ \mathcal{D}$. 
\end{proposition}
\begin{proof}
We compute:
\begin{eqnarray*}
T\circ D\circ (id_A\ot (\mu \circ T\circ D))&\overset{(\ref{tarecomp1})}{=}&
T\circ (id_A\ot (\mu \circ T))\circ \mathcal{D}\\
&\overset{(\ref{wp1})}{=}&(id_A\ot \mu )\circ \mathcal{T}\circ \mathcal{D}, 
\end{eqnarray*}
\begin{eqnarray*}
T\circ D\circ ((\mu \circ T\circ D)\ot id_A)&\overset{(\ref{tarecomp2})}{=}&
T\circ ((\mu \circ T)\ot id_A)\circ \mathcal{D}\\
&\overset{(\ref{wp2})}{=}&(\mu \ot id_A)\circ \mathcal{T}\circ \mathcal{D},
\end{eqnarray*}
finishing the proof.
\end{proof}
\begin{corollary}\label{twopstw}
Let $(\mathcal{C}, \ot )$ be a strict monoidal category and $(A, \mu )$ an algebra in
$\mathcal{C}$. Assume that $T$ and $D$ are two weak pseudotwistors for $A$, with weak 
companions $\mathcal{T}$ and respectively $\mathcal{D}$, such that the following conditions are satisfied:
\begin{eqnarray}
&&\mu \circ T\circ D=\mu \circ D\circ T, \label{comp1}\\
&&\mathcal{D}\circ (id_A\ot T)=(id_A\ot T)\circ \mathcal{D}, \label{comp2}\\
&&\mathcal{D}\circ (T\ot id_A)=(T\ot id_A)\circ \mathcal{D}. \label{comp3}
\end{eqnarray}
Then $T\circ D$ is a weak pseudotwistor for $A$, with weak companion $\mathcal{T}\circ \mathcal{D}$. 
\end{corollary}
\begin{proof}
We check (\ref{tarecomp1}), while (\ref{tarecomp2}) is similar and left to the reader:
\begin{eqnarray*}
D\circ (id_A\ot (\mu \circ T\circ D))&\overset{(\ref{comp1})}{=}&
D\circ (id_A\ot (\mu \circ D\circ T))\\
&=&D\circ (id_A\ot (\mu \circ D))\circ (id_A\ot T)\\
&\overset{(\ref{wp1})}{=}&(id_A\ot \mu )\circ \mathcal{D}\circ (id_A\ot T)\\
&\overset{(\ref{comp2})}{=}&(id_A\ot \mu )\circ (id_A\ot T)\circ \mathcal{D}\\
&=&(id_A\ot (\mu \circ T))\circ \mathcal{D}, 
\end{eqnarray*}
finishing the proof. 
\end{proof}

Let $H$ be a bialgebra over a field $k$ as in Example \ref{cocycles}, $\sigma $ (respectively $\tau $) 
a left (respectively right) 2-cocycle on $H$ and $T$ and $D$ the weak pseudotwistors 
defined in Example \ref{cocycles}. The multiplication defined on $H$ by 
\begin{eqnarray}
&&a*b=\sigma (a_1, b_1)a_2b_2\tau (a_3, b_3), \;\;\;\forall\;a, b\in H, \label{star}
\end{eqnarray}
is associative. We can obtain this as consequence of Corollary \ref{twopstw}. It is 
obvious that $T\circ D=D\circ T$, so (\ref{comp1}) is satisfied. It is easy to see that 
(\ref{comp2}) and (\ref{comp3}) are satisfied too, so $T\circ D$ is a weak pseudotwistor and 
clearly $\mu \circ T\circ D$ is exactly the multiplication $*$ defined by (\ref{star}). 

We present now another application of Proposition \ref{taretwopstw}, generalizing Remark 
4.11 in \cite{psvo}:
\begin{proposition}
Let $(\mathcal{C}, \ot )$ be a strict monoidal category, $(A, \mu )$ an algebra in
$\mathcal{C}$, $T:A\ot A\rightarrow A\ot A$ a weak pseudotwistor with weak 
companion $\mathcal{T}$ and $D_{X, Y}:X\ot Y\rightarrow X\ot Y$ a family of natural morphisms in 
$\mathcal{C}$ satisfying (\ref{laycle}). Then $T\circ D_{A, A}$ is a weak pseudotwistor for $A$. 
\end{proposition}
\begin{proof}
The naturality of $D_{X, Y}$ implies: 
\begin{eqnarray*}
&&D_{A, A}\circ (id_A\ot (\mu \circ T))=(id_A\ot (\mu \circ T))\circ D_{A, A\ot A}, \\
&&D_{A, A}\circ ((\mu \circ T)\ot id_A)=((\mu \circ T)\ot id_A)\circ D_{A\ot A, A}. 
\end{eqnarray*}
By composing on the right with $id_A\ot D_{A, A}$ and respectively $D_{A, A}\ot id_A$ we obtain:
\begin{eqnarray*}
&&D_{A, A}\circ (id_A\ot (\mu \circ T\circ D_{A, A}))=(id_A\ot (\mu \circ T))\circ D_{A, A\ot A}
\circ (id_A\ot D_{A, A}), \\
&&D_{A, A}\circ ((\mu \circ T\circ D_{A, A})\ot id_A)=((\mu \circ T)\ot id_A)\circ D_{A\ot A, A}
\circ (D_{A, A}\ot id_A).  
\end{eqnarray*}
The weak companion of $D_{A, A}$ is  
$\mathcal{D}=D_{A, A\ot A}\circ (id_A\ot D_{A, A})=D_{A\ot A, A}\circ (D_{A, A}\ot id_A)$, 
so we obtained 
(\ref{tarecomp1}) and (\ref{tarecomp2}).   
\end{proof}

Let $(\mathcal{C}, \ot )$ be a strict monoidal category, $(A, \mu )$ an algebra in
$\mathcal{C}$ and $T:A\ot A\rightarrow A\ot A$ a weak pseudotwistor. In view of Example \ref{cocycles}, 
we may think of $T$ as some sort of 2-cocycle for $A$. We will see that we can define as well some 
sort of 2-coboundaries.
\begin{proposition} \label{2coboundaries}
Let $(\mathcal{C}, \ot )$ be a strict monoidal category and $(A, \mu )$ an algebra in
$\mathcal{C}$. Assume that we are given a triple $(f, F, \mathcal{F})$, where $f:A\rightarrow A$, 
$F:A\ot A\rightarrow A\ot A$ and $\mathcal{F}:A\ot A\ot A\rightarrow A\ot A\ot A$ are 
morphisms in $\mathcal{C}$ satisfying the following conditions:
\begin{eqnarray}
&&F\circ (id_A\ot \mu )=(id_A\ot \mu )\circ \mathcal{F}, \label{cob1} \\
&&F\circ (\mu \ot id_A)=(\mu \ot id_A)\circ \mathcal{F}, \label{cob2} \\
&&f\circ \mu \circ F=\mu. \label{cob3}
\end{eqnarray}
Then the morphism $D:A\ot A\rightarrow A\ot A$, $D=F\circ (f\ot f)$ is a weak pseudotwistor with 
weak companion $\mathcal{D}:A\ot A\ot A\rightarrow A\ot A\ot A$, $\mathcal{D}=\mathcal{F}\circ 
(f\ot f\ot f)$. We denote $\partial (f, F)=D$ and call such a weak pseudotwistor 2-coboundary for $A$. 
Moreover, $f$ is an algebra homomorphism from $A^{\partial (f, F)}$ to $A$, so in particular if $f$ 
is invertible then $A^{\partial (f, F)}$ and $A$ are isomorphic as algebras. 
\end{proposition}
\begin{proof}
We check (\ref{wp1}) for $D$ and $\mathcal{D}$, while (\ref{wp2}) is similar and left to the reader: 
\begin{eqnarray*}
D\circ (id_A\ot \mu )\circ (id_A\ot D)&=&F\circ (f\ot f)\circ (id_A\ot \mu )\circ (id_A\ot F)\circ 
(id_A\ot f\ot f)\\
&=&F\circ (f\ot id_A)\circ (id_A\ot f)\circ (id_A\ot \mu )\circ (id_A\ot F)\\
&&\circ (id_A\ot f\ot f)\\
&=&F\circ (f\ot id_A)\circ (id_A\ot f\circ \mu \circ F)\circ (id_A\ot f\ot f)\\
&\overset{(\ref{cob3})}{=}&F\circ (f\ot id_A)\circ (id_A\ot \mu )\circ (id_A\ot f\ot f)\\
&=&F\circ (id_A\ot \mu )\circ (f\ot id_A\ot id_A)\circ (id_A\ot f\ot f)\\
&\overset{(\ref{cob1})}{=}&(id_A\ot \mu )\circ \mathcal{F}\circ (f\ot f\ot f)
=(id_A\ot \mu )\circ \mathcal{D}. 
\end{eqnarray*}
The fact that $f$ is an algebra homomorphism $A^{\partial (f, F)}\rightarrow A$ follows 
immediately from (\ref{cob3}). 
\end{proof}
\begin{example}
Let $(\mathcal{C}, \ot )$ be a strict monoidal category and $R_X:X\rightarrow X$ a family of 
natural isomorphisms in $\mathcal{C}$. If  
$(A, \mu )$ is an algebra in $\mathcal{C}$, then $f:=R_A$, $F:=R_{A\ot A}^{-1}$ and 
$\mathcal{F}:=R_{A\ot A\ot A}^{-1}$ satisfy the hypotheses of Proposition \ref{2coboundaries}. 
Indeed, the naturality of $R$ implies $R_{A\ot A}\circ (id_A\ot \mu )=(id_A\ot \mu )\circ R_{A\ot A\ot A}$, 
which is (\ref{cob1}), $R_{A\ot A}\circ (\mu \ot id_A)=(\mu \ot id_A)\circ R_{A\ot A\ot A}$, which is 
(\ref{cob2}), and $R_A\circ \mu =\mu \circ R_{A\ot A}$, which is (\ref{cob3}). 
\end{example}

Another example will be given in the next section. 
\section{Rota-Baxter type operators}\label{section3}
\setcounter{equation}{0}
${\;\;\;}$
We recall (see for instance the recent survey \cite{patras} and references therein) 
the concept of Rota-Baxter operator. 
Let $(A, \mu )$ be an associative algebra over a field $k$ 
and $\lambda \in k$ a fixed element. A linear map $R:A\rightarrow A$ is called 
a Rota-Baxter operator of weight $\lambda $ if it satisfies the relation
\begin{eqnarray}
&&R(a)R(b)=R(R(a)b+aR(b)+\lambda ab), \;\;\; \forall \;a, b\in A. \label{RB}
\end{eqnarray}
If this is the case, the new multiplication $*_{\lambda }$ on $A$ defined by 
\begin{eqnarray*}
&&a*_{\lambda }b=R(a)b+aR(b)+\lambda ab, \;\;\; \forall \;a, b\in A, 
\end{eqnarray*}
and called the {\em double product}, is associative and $R$ is an algebra map 
from $(A, *_{\lambda })$ to $(A, \mu )$. 
\begin{proposition} \label{RotaBaxter}
If $R:A\rightarrow A$ is a Rota-Baxter operator of weight $\lambda $ on an algebra $(A, \mu )$ as above, 
then the linear map 
\begin{eqnarray}
&&T:A\ot A\rightarrow A\ot A, \;\;\;T(a\ot b)=R(a)\ot b+a\ot R(b)+\lambda a\ot b, \;\;\;\forall \;a, b\in A, 
\label{RBpstw}
\end{eqnarray}
is a weak pseudotwistor with weak companion $\mathcal{T}:A\otimes A\otimes A 
\rightarrow A\otimes A\otimes A$, 
\begin{eqnarray*}
&&\mathcal{T}(a\ot b\ot c)=R(a)\ot R(b)\ot c+R(a)\ot b\ot R(c)
+a\ot R(b)\ot R(c)+\lambda R(a)\ot b\ot c\\
&&\;\;\;\;\;\;\;\;\;\;\;\;\;\;\;\;\;\;\;\;\;\;\;\;\;\;\;+\lambda a\ot R(b)\ot c+\lambda a\ot b\ot R(c)
+\lambda ^2a\ot b\ot c, 
\end{eqnarray*}
and the new associative product $\mu \circ T$ on $A$ coincides with the double 
product $*_{\lambda }$. 
\end{proposition}
\begin{proof}
Obviously $\mu \circ T$ coincides with $*_{\lambda }$, so we only need to prove that $T$ is a weak 
pseudotwistor. We compute, for $a, b, c\in A$:
\begin{eqnarray*}
T\circ (id_A\ot (\mu \circ T))(a\ot b\ot c)&=&T(a\ot R(b)c+a\ot bR(c)+\lambda a\ot bc)\\
&=&R(a)\ot R(b)c+a\ot R(R(b)c)+\lambda a\ot R(b)c\\
&&+R(a)\ot bR(c)+a\ot R(bR(c))+\lambda a\ot bR(c)\\
&&+\lambda R(a)\ot bc+\lambda a\ot R(bc)+\lambda ^2a\ot bc\\
&=&R(a)\ot R(b)c+\lambda a\ot R(b)c+R(a)\ot bR(c)+\lambda a\ot bR(c)\\
&&+\lambda R(a)\ot bc+\lambda ^2a\ot bc+a\ot R(R(b)c+bR(c)+\lambda bc)\\
&\overset{(\ref{RB})}{=}&R(a)\ot R(b)c+\lambda a\ot R(b)c+R(a)\ot bR(c)+\lambda a\ot bR(c)\\
&&+\lambda R(a)\ot bc+\lambda ^2a\ot bc+a\ot R(b)R(c)\\
&=&(id_A\ot \mu )\circ \mathcal{T}(a\ot b\ot c). 
\end{eqnarray*}
A similar computation shows that: 
$T\circ ((\mu \circ T)\otimes id_A)(a\ot b\ot c)=(\mu \otimes id_A)\circ \mathcal{T}(a\ot b\ot c)$. 
\end{proof}

Let $A$ be an associative algebra over a field $k$ 
and $\beta , \gamma :A\rightarrow A$ two commuting Rota-Baxter operators 
of weight $0$. It was proved in \cite{quadrialgebras} (as a consequence of the fact that, 
via $\beta $ and $\gamma $, $A$ becomes a so-called {\em quadri-algebra}) that 
the new multiplication defined on $A$ by 
\begin{eqnarray}
&&a*b=\gamma \beta (a)b+\beta (a)\gamma (b)+\gamma (a)\beta (b)+a\gamma \beta (b), \;\;\;
\forall \;a, b\in A,  \label{2RB}
\end{eqnarray}
is associative. We want to obtain this as a consequence of Corollary \ref{twopstw}. By Proposition 
\ref{RotaBaxter}, we can 
consider the 
weak pseudotwistors $T, D:A\ot A\rightarrow A\ot A$, 
\begin{eqnarray*}
&&T(a\ot b)=\gamma (a)\ot b+a\ot \gamma (b), \;\;\;D(a\ot b)=\beta (a)\ot b+a\ot \beta (b), 
\end{eqnarray*}
with weak companions $\mathcal{T}$ and respectively $\mathcal{D}$ defined by 
\begin{eqnarray*}
&&\mathcal{T}(a\ot b\ot c)=\gamma (a)\ot \gamma (b)\ot c+\gamma (a)\ot b\ot \gamma (c)+
a\ot \gamma (b)\ot \gamma (c), \\
&&\mathcal{D}(a\ot b\ot c)=\beta (a)\ot \beta (b)\ot c+\beta (a)\ot b\ot \beta (c)+
a\ot \beta (b)\ot \beta (c).
\end{eqnarray*}
Since $\gamma $ and $\beta $ commute, it is obvious that $T\circ D=D\circ T$, so (\ref{comp1}) is 
satisfied. An easy computation shows that (\ref{comp2}) and (\ref{comp3}) are 
also satisfied, so $T\circ D$ is a weak pseudotwistor and obviously 
$\mu \circ T\circ D$ is exactly the multiplication (\ref{2RB}). 
\begin{example}
Let $A$ be an associative algebra over a field $k$ and $R:A\rightarrow A$ a bijective Rota-Baxter 
operator of weight $\lambda $, with inverse $R^{-1}$. Then the linear maps $f:=R$, 
$F:A\ot A\rightarrow A\ot A$, $F(a\ot b)=a\ot R^{-1}(b)+R^{-1}(a)\ot b+\lambda R^{-1}(a)\ot R^{-1}(b)$, 
and $\mathcal{F}:A\ot A\ot A\rightarrow A\ot A\ot A$, $\mathcal{F}(a\ot b\ot c)=R^{-1}(a)\ot b\ot c+
a\ot R^{-1}(b)\ot c+a\ot b\ot R^{-1}(c)+\lambda a\ot R^{-1}(b)\ot R^{-1}(c)+
\lambda R^{-1}(a)\ot b\ot R^{-1}(c)+\lambda R^{-1}(a)\ot R^{-1}(b)\ot c+
\lambda ^2R^{-1}(a)\ot R^{-1}(b)\ot R^{-1}(c)$, satisfy the hypotheses of Proposition 
\ref{2coboundaries}, as one can easily check, and the weak pseudotwistor  $\partial (f, F)$ 
(notation as in Proposition \ref{2coboundaries}) 
is the 
one defined by (\ref{RBpstw}). 
\end{example}
\begin{remark}
Bijective Rota-Baxter operators exist. For example, if $A$ is an associative algebra over a field $k$ 
and $\lambda \in k$, $\lambda \neq 0$, then the (bijective) linear map $R:A\rightarrow A$, 
$R(a)=-\lambda a$, for all $a\in A$, is a Rota-Baxter operator of weight $\lambda $, cf. \cite{guodifferential}. 
\end{remark}
\begin{definition}
Let $A$ be an associative algebra over a field $k$ and $\alpha , \beta :A\rightarrow A$ two linear maps. 
A linear map $R:A\rightarrow A$ will be called an {\em $(\alpha , \beta )$-Rota-Baxter operator} if 
the following conditions are satisfied, for all $a, b\in A$:
\begin{eqnarray*}
&&\alpha (R(a)R(b))=\alpha (R(a))\alpha (R(b)), \\
&&\beta (R(a)R(b))=\beta (R(a))\beta (R(b)), \\
&&R(a)R(b)=R(\alpha (R(a))b+a\beta (R(b))).
\end{eqnarray*}
\end{definition}

Obviously, an $(id_A, id_A)$-Rota-Baxter operator is just a Rota-Baxter operator of weight $0$. 
A nontrivial example (which actually inspired this concept) may be found in \cite{twisteddendriform}: 
$A$ is the algebra of continuous functions on $\mathbb{R}$ with values in some unital algebra $B$, 
$q$ is a number with $0<q<1$, $\alpha =id_A$, $\beta =M_q$ is the $q$-dilation operator and 
$R=I_q$ is the Jackson $q$-integral. Then formula (20) in \cite{twisteddendriform} says 
exactly that $R$ is an $(\alpha , \beta )$-Rota-Baxter operator. 

\begin{proposition}
If $R$ is an $(\alpha , \beta )$-Rota-Baxter operator as above, 
then the linear map 
\begin{eqnarray*}
&&T:A\ot A\rightarrow A\ot A, \;\;\;T(a\ot b)=\alpha (R(a))\ot b+a\ot \beta (R(b)), \;\;\;\forall \;a, b\in A, 
\end{eqnarray*}
is a weak pseudotwistor with weak companion $\mathcal{T}:A\otimes A\otimes A 
\rightarrow A\otimes A\otimes A$, 
\begin{eqnarray*}
&&\mathcal{T}(a\ot b\ot c)=\alpha (R(a))\ot \alpha (R(b))\ot c+\alpha (R(a))\ot b\ot \beta (R(c)) 
+a\ot \beta (R(b))\ot \beta (R(c)).
\end{eqnarray*}
Consequently, the new multiplication defined on $A$ by the formula $a*b=\alpha (R(a))b+a\beta (R(b))$, 
for all $a, b\in A$, is associative. 
\end{proposition}
\begin{proof}
Follows by a direct computation.
\end{proof}
\begin{example}
Let $A$ be an associative algebra over a field $k$ and $R:A\rightarrow A$  a so-called 
Reynolds operator (see for instance \cite{uchinotwisting}), that is $R$ satisfies the following condition:
\begin{eqnarray*}
&&R(a)R(b)=R(R(a)b+aR(b)-R(a)R(b)), \;\;\; \forall \;a, b\in A. 
\end{eqnarray*}
If one defines a new multiplication on $A$, by 
\begin{eqnarray*}
&&a*b=R(a)b+aR(b)-R(a)R(b), \;\;\; \forall \;a, b\in A, 
\end{eqnarray*}
then (for instance as a consequence of the theory developped in \cite{uchinoquantum}) $*$
is associative. 

If we define the linear map 
\begin{eqnarray*}
&&T:A\ot A\rightarrow A\ot A, \;\;\;T(a\ot b)=R(a)\ot b+a\ot R(b)-R(a)\ot R(b), \;\;\;\forall \;a, b\in A, 
\end{eqnarray*}
then one can check, by a direct computation, that $T$ 
is a weak pseudotwistor, with weak companion $\mathcal{T}:A\otimes A\otimes A 
\rightarrow A\otimes A\otimes A$, 
$\mathcal{T}(a\ot b\ot c)=R(a)\ot R(b)\ot c+R(a)\ot b\ot R(c)
+a\ot R(b)\ot R(c)-2R(a)\ot R(b)\ot R(c)$,
and the resulting associative multiplication $\mu \circ T$ coincides with $*$. 
\end{example}

We recall from \cite{leroux1} that, if $(A, \mu )$ is an associative unital algebra with unit $1_A$ over a 
field $k$, 
a linear map $P:A\rightarrow A$ is called a {\em TD-operator} if 
\begin{eqnarray}
P(a)P(b)=P(P(a)b+aP(b)-aP(1_A)b), \;\;\;\forall \; a, b\in A. \label{TDop}
\end{eqnarray}
If this is the case, the new multiplication defined on $A$ by $a*b=P(a)b+aP(b)-aP(1_A)b$, 
for all $a, b\in A$, is associative. 
\begin{proposition}
If $P:A\rightarrow A$ is a TD-operator, then the linear map $T:A\ot A\rightarrow A\ot A$, 
$T(a\ot b)=P(a)\ot b+a\ot P(b)-aP(1_A)\ot b$, for all $a, b\in A$, is a weak pseudotwistor with weak companion 
$\mathcal{T}:A\ot A\ot A\rightarrow A\ot A\ot A$, $\mathcal{T}(a\ot b\ot c)=P(a)\ot P(b)\ot c+
P(a)\ot b\ot P(c)+a\ot P(b)\ot P(c)+aP(1_A)\ot bP(1_A)\ot c 
-aP(1_A)\ot P(b)\ot c-aP(1_A)\ot b\ot P(c)-P(a)\ot bP(1_A)\ot c$, and the associative multiplications $*$ 
and $\mu \circ T$ coincide. 
\end{proposition}
\begin{proof}
A straightforward computation, by using also the identity $P(1_A)P(a)=P(a)P(1_A)$, for all 
$a\in A$, which follows immediately from (\ref{TDop}). 
\end{proof}

We recall from \cite{leroux2} that, if $(A, \mu )$ is an associative algebra over a 
field $k$, a right (respectively left) Baxter operator on $A$ is a linear map $P:A\rightarrow A$ 
(respectively $Q:A\rightarrow A$) such that $P(a)P(b)=P(P(a)b)$ (respectively $Q(a)Q(b)=Q(aQ(b))$, 
for all $a, b\in A$. If moreover $P$ and $Q$ commute, then, by \cite{leroux2}, Theorem 2.10, the 
new multiplication defined on $A$ by $a*b=P(a)Q(b)$, for all $a, b\in A$, is associative. By a straightforward 
computation, one proves the following result: 
\begin{proposition}
The linear map $T:A\ot A\rightarrow A\ot A$, $T=P\ot Q$, is a weak pseudotwistor, with weak companion 
$\mathcal{T}:A\ot A\ot A\rightarrow A\ot A\ot A$, $\mathcal{T}=P\ot P\circ Q\ot Q$, 
and the associative multiplications $*$ 
and $\mu \circ T$ coincide. 
\end{proposition}
\section{An equivalence relation}
\setcounter{equation}{0}
\begin{remark}
Let $(\mathcal{C}, \ot )$ be a strict monoidal category and 
$A$ an algebra in $\mathcal{C}$. Then $T=id_{A\ot A}$ is a weak pseudotwistor for $A$, with weak 
companion $\mathcal{T}=id_{A\ot A\ot A}$, and $A^T=A$. 
\end{remark}
\begin{proposition}
Let $(\mathcal{C}, \ot )$ be a strict monoidal category and 
$(A, \mu )$ an algebra in $\mathcal{C}$. 
Let $T:A\ot A\rightarrow A\ot A$ be a weak pseudotwistor for $A$ with weak companion $\mathcal{T}$ 
and $D:A^T\ot A^T\rightarrow A^T\ot A^T$ a weak pseudotwistor for $A^T$ with weak 
companion $\mathcal{D}$. Then $T\circ D$ is a weak pseudotwistor for $A$ with weak companion 
$\mathcal{T}\circ \mathcal{D}$, and $(A^T)^D=A^{T\circ D}$. 
\end{proposition}
\begin{proof}
We prove (\ref{wp1}), while (\ref{wp2}) is similar and left to the reader: 
\begin{eqnarray*}
T\circ D\circ (id_A\ot (\mu \circ T\circ D))&=&T\circ (D\circ (id_A\ot ((\mu \circ T)\circ D)))\\
&\overset{(\ref{wp1})}{=}&T\circ (id_A\ot (\mu \circ T))\circ \mathcal{D}\\
&\overset{(\ref{wp1})}{=}&(id_A\ot \mu )\circ \mathcal{T}\circ \mathcal{D}.
\end{eqnarray*}
The fact that $(A^T)^D=A^{T\circ D}$ is obvious. 
\end{proof}
\begin{proposition}
Let $(\mathcal{C}, \ot )$ be a strict monoidal category, 
$(A, \mu )$ an algebra in $\mathcal{C}$ and 
$T:A\ot A\rightarrow A\ot A$ a weak pseudotwistor for $A$ with weak companion $\mathcal{T}$, such that 
$T$ and $\mathcal{T}$ are invertible. Then $T^{-1}$ is a weak pseudotwistor for $A^T$ 
with weak companion $\mathcal{T}^{-1}$, and $(A^T)^{T^{-1}}=A$.
\end{proposition}
\begin{proof}
We prove (\ref{wp1}) and leave (\ref{wp2}) to the reader. We need to prove that 
\begin{eqnarray*}
&&T^{-1}\circ (id_A\ot ((\mu \circ T)\circ T^{-1}))=(id_A\ot (\mu \circ T))\circ \mathcal{T}^{-1}.
\end{eqnarray*}
This is obviously equivalent to 
\begin{eqnarray*}
&&T\circ (id_A\otimes (\mu \circ T))=(id_A\otimes \mu )\circ \mathcal{T}.
\end{eqnarray*}
The fact that $(A^T)^{T^{-1}}=A$ is obvious. 
\end{proof}
\begin{proposition}
Let $(\mathcal{C}, \ot )$ be a strict monoidal category, 
$(A, \mu _A)$ and $(B, \mu _B)$ two algebras in $\mathcal{C}$, $f:A\rightarrow B$ an algebra 
isomorphism and  
$T:A\ot A\rightarrow A\ot A$ a weak pseudotwistor for $A$ with weak companion $\mathcal{T}$. Then 
$D:=(f\ot f)\circ T\circ (f^{-1}\ot f^{-1})$ is a weak pseudotwistor for $B$ with weak companion 
$\mathcal{D}:=(f\ot f\ot f)\circ \mathcal{T}\circ (f^{-1}\ot f^{-1}\ot f^{-1})$, and $f$ is also an 
algebra isomorphism from $A^T$ to $B^D$. 
\end{proposition}
\begin{proof}
We prove (\ref{wp1}) for $D$ and leave (\ref{wp2}) to the reader:
\begin{eqnarray*}
D\circ (id_B\ot (\mu _B\circ D))&=&D\circ (id_B\ot (\mu _B\circ (f\ot f)\circ T\circ (f^{-1}\ot f^{-1})))\\
&=&D\circ (id_B\ot (f\circ \mu _A\circ T\circ (f^{-1}\ot f^{-1})))\\
&=&(f\ot f)\circ T\circ (f^{-1}\ot f^{-1})\circ (id_B\ot (f\circ \mu _A\circ T\circ (f^{-1}\ot f^{-1})))\\
&=&(f\ot f)\circ T\circ ((f^{-1}\circ id_B)\ot (f^{-1}\circ f\circ \mu _A\circ T\circ (f^{-1}\ot f^{-1})))\\
&=&(f\ot f)\circ T\circ (f^{-1}\ot (\mu _A\circ T\circ (f^{-1}\ot f^{-1})))\\
&=&(f\ot f)\circ T\circ ((id_A\circ f^{-1})\ot (\mu _A\circ T\circ (f^{-1}\ot f^{-1})))\\
&=&(f\ot f)\circ T\circ (id_A\ot (\mu _A\circ T))\circ (f^{-1}\ot f^{-1}\ot f^{-1})\\
&\overset{(\ref{wp1})}{=}&(f\ot f)\circ (id_A\otimes \mu _A)\circ \mathcal{T}\circ 
(f^{-1}\ot f^{-1}\ot f^{-1})\\
&=&((f\circ id_A)\ot (f\circ \mu _A))\circ \mathcal{T}\circ 
(f^{-1}\ot f^{-1}\ot f^{-1})\\
&=&(f\ot (\mu _B\circ (f\ot f)))\circ \mathcal{T}\circ 
(f^{-1}\ot f^{-1}\ot f^{-1})\\
&=&((id_B\circ f)\ot (\mu _B\circ (f\ot f)))\circ \mathcal{T}\circ 
(f^{-1}\ot f^{-1}\ot f^{-1})\\
&=&(id_B\ot \mu _B)\circ (f\ot f\ot f)\circ \mathcal{T}\circ 
(f^{-1}\ot f^{-1}\ot f^{-1})\\
&=&(id_B\ot \mu _B)\circ \mathcal{D}, \;\;\;q.e.d.
\end{eqnarray*}
The fact that f is an algebra morphism from $A^T$ to $B^D$ follows from the fact that 
$f\circ \mu _A\circ T=\mu _B\circ (f\ot f)\circ T=\mu _B\circ D\circ (f\ot f)$. 
\end{proof}
\begin{definition}
Let $(\mathcal{C}, \ot )$ be a strict monoidal category and  
$(A, \mu _A)$, $(B, \mu _B)$ two algebras in $\mathcal{C}$. We will say that $A$ and $B$ are 
{\em twist equivalent}, and write $A\equiv _tB$, if there exists an invertible weak pseudotwistor 
$T$ for $A$, with invertible weak companion $\mathcal{T}$, such that $A^T$ and $B$ are isomorphic as algebras. 
\end{definition}
\begin{remark}
In view of Remark \ref{echival}, we have $A\equiv _tB$ if and only if there exists an invertible pseudotwistor 
$T$ for $A$, with invertible companions $\tilde{T}_1$ and $\tilde{T}_2$, such that $A^T$ and $B$ are isomorphic as algebras, if and only if there exists an invertible $R$-matrix  
$T$ for $A$, with invertible companions $\overline{T}_1$ and $\overline{T}_2$, 
such that $A^T$ and $B$ are isomorphic as algebras.
\end{remark} 

Obviously, two isomorphic algebras are twist equivalent. 

As a consequence of the above results, we obtain:
\begin{proposition}
$\equiv _t$ is an equivalence relation. 
\end{proposition}
\begin{example}
In the setting of Example \ref{cocycles}, if $\sigma $ (respectively $\tau $) is a convolution 
invertible left (respectively right) 2-cocycle, then $_{\sigma }H\equiv _tH$ and $ H_{\tau }\equiv _tH$. 
\end{example}
\begin{example}
Let $(A, \mu _A)$ and $(B, \mu _B)$ be two associative algebras over a field $k$ and 
$R:B\ot A\rightarrow A\ot B$ a twisting map, with Sweedler-type notation $R(b\ot a)=a_R\ot b_R$, 
for $a\in A$, $b\in B$. We can consider the twisted tensor product $A\ot _RB$ 
(cf. \cite{cap}, \cite{vandaele}), which is the associative algebra structure on the linear space $A\ot B$ given 
by the multiplication $(a\ot b)(a'\ot b')=aa'_R\ot b_Rb'$, for $a, a'\in A$, $b, b'\in B$. Define the linear map 
\begin{eqnarray*}
&&T:(A\ot B)\ot (A\ot B)\rightarrow (A\ot B)\ot (A\ot B), \\
&&T((a\ot b)\ot (a'\ot b'))=(a\ot b_R)\ot (a'_R\ot b'). 
\end{eqnarray*}
By \cite{lpvo}, $T$ is a so-called twistor for the associative algebra $A\ot B$, in particular it is a weak 
pseudotwistor with weak companion 
\begin{eqnarray*}
&&\mathcal{T}:(A\ot B)\ot (A\ot B)\ot (A\ot B)\rightarrow (A\ot B)\ot (A\ot B)\ot (A\ot B), \\
&&\mathcal{T}((a\ot b)\ot (a'\ot b')\ot (a''\ot b''))=(a\ot (b_R)_{\mathcal{R}})\ot (a'_R\ot b'_r)\ot 
((a''_r)_{\mathcal{R}}\ot b''), 
\end{eqnarray*}
where $r$ and $\mathcal{R}$ are two more copies of $R$; moreover, we have that $A\ot _RB=(A\ot B)^T$. 

Assume now that $R$ is a bijective map. Then obviously $T$ and $\mathcal{T}$ are also bijective, 
hence $A\ot _RB\equiv _tA\ot B$. 
\end{example}

{\bf Note added.} We used the term ''Rota-Baxter type operator'' in an informal way, to designate 
an operator that is ''similar'' to a Rota-Baxter operator. Professor Li Guo kindly draw out attention to the 
paper \cite{RBTO}, where this term was introduced as a rigorous concept and moreover a conjectural 
list of possible Rota-Baxter type operators was proposed. 
\begin{center}
ACKNOWLEDGEMENTS
\end{center}
We would like to thank the referee for some useful suggestions. 

\end{document}